\documentclass[leqno,12pt,a4paper]{amsart}
\usepackage[letterpaper,margin=1in,footskip=0.25in]{geometry}
\usepackage{amscd,setspace,amssymb,amsopn,amsmath,amsthm,mathrsfs,lmodern,graphics,amsfonts,enumerate,verbatim,calc,times} 
\usepackage[all]{xy}
\usepackage[para]{threeparttable}
\usepackage[tableposition=top]{caption}
\usepackage{booktabs}
\usepackage{soul}
\usepackage[colorlinks=true, citecolor=RoyalBlue, filecolor=RoyalBlue, linkcolor=RoyalBlue, urlcolor=RoyalBlue, hyperindex]{hyperref}
\usepackage[dvipsnames]{xcolor}
\usepackage[referable]{threeparttablex}
\usepackage{bigdelim,multirow}
\usepackage{longtable}
\usepackage{enumitem}
\usepackage{pdflscape}
\usepackage{newpxtext}
\usepackage{cabin}
\usepackage[varqu,varl]{inconsolata}
\usepackage[bigdelims,vvarbb]{newpxmath}
\usepackage{todonotes}
\usepackage{footnote}
\usepackage{tikz-cd}
\usepackage{pgfplots}
\usepackage{crossreftools}
\usepackage[capitalise]{cleveref} 
\pgfplotsset{width=10cm,compat=1.9}
\usepackage{arydshln}
\DeclareTextFontCommand{\textcyr}{\cyr}
\usetikzlibrary{spy}
\usepackage{amssymb,amsmath}
\DeclareFontFamily{OT1}{rsfs}{}	
\DeclareFontShape{OT1}{rsfs}{n}{it}{<-> rsfs10}{}
\DeclareMathAlphabet{\fmathscr}{OT1}{rsfs}{n}{it}
\newtheorem{Theoremx}{Theorem}
 
\newtheorem{theorem}{Theorem}[section]
\theoremstyle{definition}

\theoremstyle{definition}

\newtheorem{lemma}[theorem]{Lemma}

\newtheorem{proposition}[theorem]{Proposition}

\newtheorem{question}{Question}

\theoremstyle{definition}
\newtheorem{definition}[theorem]{Definition}

\theoremstyle{remark}

\renewcommand{\ker}{\operatorname{ker}}

\usepackage{stmaryrd}

\newcommand{\Height}{\operatorname{ht}}

\newcommand{\los}{\L o\'s's Theorem}

\newcommand{\nil}{\operatorname{nil}}

\newcommand{\N}{\mathbb{N}}

\newcommand{\Z}{\mathbb{Z}}

\newcommand{\sk}{\mathscr{k}}

\newcommand{\fm}{\mathfrak{m}}
\newcommand{\fp}{\mathfrak{p}}

\newcommand{\cF}{\mathcal{F}}

\newcommand{\cP}{\mathcal{P}}

\newcommand{\Tot}{\operatorname{Tot}}

\newcommand{\wh}{\widehat}
\DeclareMathOperator{\ord}{ord}
\DeclareMathOperator{\fte}{fte}
\DeclareMathOperator*{\ulim}{ulim}
\newcommand{\gdim}{\operatorname{gdim}}
\usepackage{titlesec}		
\titleformat{\section}[block]{\large\scshape\bfseries\filcenter}{\thesection.}{1em}{}
\titleformat{\subsection}[hang]{\large\scshape\bfseries}{\thesubsection}{1em}{}	
\titleformat{\subsubsection}[hang]{\large\scshape\bfseries}{\thesubsubsection}{1em}{}

\crefname{item}{Item}{Items}
\crefname{equation}{Equation}{Equations}

\makeatletter
\newcommand{\leqnomode}{\tagsleft@true\let\veqno\@@leqno}
\newcommand{\reqnomode}{\tagsleft@false\let\veqno\@@eqno}
\makeatother

\makeatletter
\newcommand{\colim@}[2]{%
  \vtop{\m@th\ialign{##\cr
    \hfil$#1\operator@font colim$\hfil\cr
    \noalign{\nointerlineskip\kern1.5\ex@}#2\cr
    \noalign{\nointerlineskip\kern-\ex@}\cr}}%
}
\newcommand{\colim}{%
  \mathop{\mathpalette\colim@{\rightarrowfill@\textstyle}}\nmlimits@
}
\makeatother

\usepackage[backend=biber, style=alphabetic, sorting=anyt,maxalphanames=5,maxbibnames=99]{biblatex}
\usepackage{hyperref}
\usepackage{xurl}
\hypersetup{breaklinks=true}
\renewbibmacro{in:}{}

\begin{document}

\title{Uniform arithmetic in local rings via ultraproducts}
\thanks{The SMALL REU was funded by NSF Grants DMS \#2241623 and DMS \#1947438}

\author[Adams]{Clay Adams}
\address{Department of Mathematics, Harvey Mudd College, Claremont, CA 91711}
\email{ccadams@hmc.edu}

\author[Cantor]{Francesca Cantor}
\address{Mathematics and Statistics Department, Swarthmore College, Swarthmore, PA 19081}
\email{fcantor1@swarthmore.edu}

\author[Gashi]{Anese Gashi}
\address{Department of Mathematics, Williams College, Williamstown, MA 01267}
\email{ag25@williams.edu}

\author[Mujevic]{Semir Mujevic}
\address{Department of Mathematics, Brown University, Providence, RI 02912}
\email{semir\textunderscore mujevic@brown.edu}

\author[Park]{Sejin Park}
\address{Department of Mathematics, Brown University, Providence, RI 02912}
\email{sejin\textunderscore park@brown.edu}

\author[Simpson]{Austyn Simpson}
\thanks{Simpson was supported by NSF postdoctoral fellowship DMS \#2202890.}
\address{Department of Mathematics, University of Michigan, Ann Arbor, MI 48109 USA}
\email{austyn@umich.edu}
\urladdr{\url{https://austynsimpson.github.io}}

\author[Zomback]{Jenna Zomback}
\address{Department of Mathematics, University of Maryland, College Park, MD 20742}
\email{zomback@umd.edu}
\urladdr{\url{sites.google.com/umd.edu/zomback}}

\maketitle
\vspace{-.5cm}
\begin{abstract}
    We reinterpret various properties of Noetherian local rings via the existence of some $n$-ary numerical function satisfying certain uniform bounds. We provide such characterizations for seminormality, weak normality, generalized Cohen--Macaulayness, and $F$-purity, among others. Our proofs that such numerical functions exist are nonconstructive and rely on the transference of the property in question from a local ring to its ultrapower or catapower.
\end{abstract}

\section{Introduction}
The present article is broadly concerned with providing uniform and elementary interpretations for a Noetherian local ring to satisfy certain properties. Through an appropriate lens, this program was arguably initiated by Rees who gave the following characterization of \emph{analytic irreducibility}:

\begin{theorem}\cite{Ree61} Let $(R, \fm)$ be a Noetherian local ring. Then its $\fm$-adic completion $\wh R$ is a \texttt{domain} if and only if there exists a binary numerical function $\varphi_R$ such that for all $x, y\in R$ the following inequality holds:
    \begin{align*}
        \ord_R(xy) \leq \varphi_R(\ord_R(x), \ord_R(y)).
    \end{align*} 
\end{theorem}
Here and in the sequel, $\ord_R(x):=\sup\{n\mid a\in\fm^n\}$ is the \emph{$\fm$-adic order} of an element $x$ in a Noetherian local ring $(R,\fm)$ and an \emph{$n$-ary numerical function} is a function $\varphi: (\N\cup \{\infty\})^n\to (\N\cup\{\infty\})$ such that $\varphi(a_1, \ldots, a_n) = \infty$ if and only if $a_i = \infty$ for some $i$. This classical theorem has since been refined to ensure that one may take $\varphi_R$ to be of the form $\varphi_R(-,-)=C_R\max\{-,-\}$ for some constant $C_R>0$ \cite[Proposition 2.2]{HS01}. The present work is an attempt at understanding the extent to which other properties of local rings may be similarly classified.

If one is willing to forego constructive proofs (and in particular linear bounds as in the above), this goal can in fact be quite tractable via a package of tools pioneered by Schoutens in \cite[\S 12]{Sch13} (see also \cite[\S 8.2.2]{Sch10}) centered around \emph{ultraproducts}. With these nonstandard methods, Schoutens classified in \emph{op. cit.} many familiar properties such as reducedness, Cohen--Macaulayness, and normality via the existence of some $n$-ary function. We continue this line of inquiry here, and our results fall into three themes: \begin{enumerate}
    \item variants of normality (\cref{sec: variants of normality});
    \item Frobenius singularities (\cref{sec: F-singularities});
    \item variants of the Cohen--Macaulay property (\cref{sec: Variants of CM}).
\end{enumerate}   In the first category, we give characterizations for seminormality and weak normality, and we choose to highlight the first to give a flavor for our results in this thread.

\begin{Theoremx}(= \cref{thm: phi characterization for seminormal})
    Let $(R,\fm)$ be a $d$-dimensional excellent local ring. Then $R$ is \texttt{seminormal} if and only if there exists a binary numerical function $\varphi_R$ such that for every $y,z\in R$,
\begin{align*}
    \min\{\ord_{R/z^2R} y^2, \ord_{R/z^3 R} y^3\}\leq \varphi_R(\ord_{R/zR}(y), \deg(z)).
\end{align*}
\end{Theoremx}
Here, $\deg(z)$ is the minimum value of $\ell_R(R/(z,w_2,\dots,w_d))$ ranging over all $d-1$-tuples $w_2,\dots, w_d\in\fm$.

In a separate direction, we collect several similar results of this flavor for $F$-nilpotent, (Cohen--Macaulay) $F$-injective, and $F$-pure rings, motivated by the existing work for $F$-rationality and weak $F$-regularity (see \cite[Theorems 12.17, 12.18]{Sch13}). For instance, we prove the following regarding $F$-pure and $F$-nilpotent rings:

\begin{Theoremx}
    Let $(R,\fm)$ be a Noetherian local ring of prime characteristic $p>0$.
    \begin{enumerate}
    \item (= \cref{thm: F-nilpotent numerical function}) If $R$ is excellent and equidimensional, then $R$ is \texttt{$F$-nilpotent} if and only if there exists a ternary numerical function $\varphi_R: (\N \cup \{\infty \})^3 \rightarrow \N \cup \{\infty \}$ such that for all $x,y \in R$ and every parameter ideal of $I \subseteq R$, 
    \[ \min_e \{\ord_{R/I^{[p^e]}}(xy^{p^e})\} \leq \varphi_R(\deg(x), \ell_{R}(R/I^F), \ord_{R/I^F}(y)). \]
    \item (= \cref{thm: F-pure phi function}) If $R$ is approximately Gorenstein, then $R$ is \texttt{$F$-pure} if and only if there exists a binary numerical function $\varphi_R: (\N \cup \{\infty \})^2 \rightarrow \N \cup \{\infty \}$ such that for all $x \in R$ and every $\fm$-primary ideal of $I \subseteq R$, 
    \[ \min_{e\geq \fte(I)} \{\ord_{R/I^{[p^e]}}(x^{p^e})\} \leq \varphi_R(\ord_{R/I}(x), \fte(I),\ell_R(R/I)). \]
    \end{enumerate}
\end{Theoremx}

Finally, in \cref{sec: Variants of CM} we characterize the Buchsbaum (stated below) and generalized Cohen--Macaulay properties via a numerical function whose input in part is given by colengths of systems of parameters:

\begin{Theoremx}\label{thm: thm-C}(= \Cref{thm: phi characterization of Buchsbaum})
    Let $(R,\fm)$ be a $d$-dimensional local ring. Then $R$ is \texttt{Buchsbaum} if and only if there exists a binary numerical function $\varphi_R$ such that for any system of parameters $\mathbf{x}=x_1,\dots, x_d\in R$ and any $d$-tuple $\mathbf{y}=y_1,\dots, y_d\in R$,
\begin{align*}\ord_R(x_1y_1+\cdots+x_dy_d)\leq\varphi_R(\ell_R(R/\mathbf{x}R),\ord_{R/(x_1,\dots, x_{d-1}):\fm}(y_d)).
    \end{align*}
\end{Theoremx}
This builds on analogous work for the classical Cohen--Macaulay property \cite[Theorem 12.14]{Sch13}. A comprehensive list taken from the present article and \cite[\S 12]{Sch13} may be found in \cref{table: uniform arithmetic characterizations}.

\subsection{Sketch of proof technique}
Our method of proof takes inspiration from \cite[\S 12]{Sch13} -- we employ the ultrapower (and a certain quotient thereof called the catapower) of a Noetherian local ring to produce the above numerical functions. A skeleton for the common thread of each proof is as follows. We are interested in showing that a local ring $(R,\fm,\sk)$ satisfies some property $\cP$ if and only if there exists an $m$-ary numerical function $\varphi_R$ such that for all $\underline{x}=x_1,\dots, x_s\in R$, all ideals $\underline{I}=I_1,\dots, I_t$, 
\begin{align}
F(\underline{x},\underline{I})\leq\varphi_R(f_1(\underline{x},\underline{I}),\dots, f_m(\underline{x},\underline{I}))\label{eq:intro-summary}
\end{align}
where $f_i$ and $F$ are numerical values associated to ideals and ring elements. These values are often some assortment of $\fm$-adic order modulo some ideal, degree, and colength. We suppose that no such numerical function exists, and produce a tuple $\underline{a}=(a_1,\dots, a_m)\in\N^m$ such that for every $n>0$ there are sequences of ($s$-tuples of) elements $\underline{x}_n$ and ($t$-tuples of) ideals $\underline{I}_n$ constituting a counterexample (i.e. for which $f_i(\underline{x}_n,\underline{I}_n)=a_i$ but for which $F(\underline{x}_n,\underline{I}_n)>n$). We obtain associated tuples $\underline{x}_\natural$ and $\underline{I}_\natural$ in the ultrapower $R_\natural$, a quasi-local (often non-Noetherian) ring which is summarized in \cref{sec: preliminaries}. We then view their images in the catapower $R_\sharp$, which is a (Noetherian local) quotient of $R_\natural$ and in fact is geometrically regular over $R$. The functions $F$ and $f_i$ are designed to produce a contradiction based on the given property $\cP$. It is critical both for $\cP$ to ascend to $R_\sharp$ and for $f_i(\underline{x}_\natural,\underline{I}_\natural)=a_i$ for this purpose.

\subsection{Notational conventions} A \emph{quasi-local} ring $(R,\fm)$ is a commutative ring with identity and unique maximal ideal, and is not necessarily Noetherian. A \emph{local ring} is taken to mean a Noetherian quasi-local ring. We denote by $\Tot(R)$ the total ring of fractions of $R$. We denote by $R^\circ:= R\setminus \bigcup\limits_{\fp\in\min(R)}$ the complement of the union of the minimal primes of $R$.

\subsection*{Acknowledgments}
We owe an intellectual debt to Hans Schoutens for the ideas introduced in \cite{Sch13}. This project was completed at the 2023 SMALL REU at Williams College, which we thank Steven J. Miller for organizing.

\section{Preliminaries}\label{sec: preliminaries}

\subsection{Background on ultraproducts}
Throughout this paper we utilize \textit{ultraproducts} and \textit{cataproducts}, and we give a brief summary in this subsection. The underlying tool for both of these objects is a non-principal ultrafilter, which captures what it means for a subset of $\N$ to be large.
\begin{definition}
    Let $\cF\subseteq 2^\N$, the power set of $\N$. Consider the following axioms:
    \begin{enumerate}
        \item[{$(\hyperref[F1]{\textup{F}_{\textup{I}}})$}] $\varnothing\notin\mathcal F$.\label{F1}
        \item[{$(\hyperref[F2]{\textup{F}_{\textup{II}}})$}] If $A\in\mathcal F$ and $A\subseteq B$, then $B\in\mathcal F$.\label{F2}
        \item[{$(\hyperref[F3]{\textup{F}_{\textup{III}}})$}] If $A, B\in \mathcal F$, then $A\cap B\in\mathcal F$.\label{F3}
         \item[{$(\hyperref[UF]{\textup{UF}})$}] For every $A\subseteq \N$, either $A\in\mathcal F$ or $I - A\in\mathcal F$.\label{UF}
         \item[{$(\hyperref[NP]{\textup{NP}})$}] If $A\subseteq \N$ is finite, then $A\notin\mathcal F$.\label{NP}
     \end{enumerate}
 A subset $\mathcal F$ satisfying $(\hyperref[F1]{\textup{F}_{\textup{I}}})$ -- $(\hyperref[F3]{\textup{F}_{\textup{III}}})$ is known as a \textit{filter}. A filter additionally satisfying $(\hyperref[UF]{\textup{UF}})$ is known as an \textit{ultrafilter}, and an ultrafilter satisfying $(\hyperref[NP]{\textup{NP}})$ is a \textit{non-principal ultrafilter}.
\end{definition}
The existence of non-principal ultrafilters follows from the axiom of choice (see \cite[Chapter 2]{Gol98}). The methods and results of the present article do not depend on the non-principal ultrafilter, so we fix once and for all such an object $\cF$ and usually suppress further reference to it in the sequel (except when it is notationally useful). We are now able to define the ultraproduct of a collection of sets.
\begin{definition}\label{def: ultraproduct}
    Let $\cF$ be a fixed non-principal ultrafilter on $\N$, and let $\{A_i\}_{i\in \N}$ be a collection of sets. The \emph{ultraproduct} of the $A_i$, denoted interchangeably by $\ulim_{i\to \cF}A_i$, $\ulim(A_i)$, or $A_\natural$, is the quotient of the direct product $\prod_{i\in \N}A_i$ by the equivalence relation
    \begin{align*}
    (a_i)_{i\in \N} \sim (b_i)_{i\in \N} \,\,\,\text{ iff }\,\,\, \{i\in \N\mid a_i = b_i\}\in\mathcal F.
\end{align*}
We denote an equivalence class of the above relation interchangeably by $\ulim a_i$ or $a_\natural$. In the case that $\{i\in\N\mid A_i=A\}\in\cF$ for some fixed $A$, then we call $A_\natural$ the \textit{ultrapower} of $A$.
\end{definition}

The ultraproduct preserves many properties of the constituent sets. For instance, if each $A_i$ has a binary operation $(a_i,b_i)\mapsto \phi_i(a_i,b_i)$, there is a well-defined induced binary operation on $A_\natural$ via $(\ulim a_i,\ulim b_i)\mapsto \ulim(\phi_i(a_i,b_i))$. In particular, if each $A_i$ is a group, ring, or field, so is $A_\natural$. If each $A_i$ is an $R_i$-algebra where each $R_i$ is a ring, then $A_\natural$ is naturally an $R_\natural$-algebra. Moreover, if $I_i\subseteq A_i$ are ideals, then $I_\natural$ is an ideal of $A_\natural$. The reader may find a collection of basic facts about ultraproducts of rings in \cite[Chapter 2]{Sch10}, and we freely use much of this material without justification.

The following theorem due to \L o\'s allows us to understand first order properties of an ultraproduct by understanding those of its constituents.
\begin{theorem}[\los]\cite[Theorem 2.3.2]{Sch10}
    Let $R$ be a ring and let $\{A_i\}$ be a family of $R$-algebras. If $\varphi$ is a first order sentence with parameters from $R$, then $\varphi$ holds in almost all $A_i$ if and only if $\varphi$ holds in the ultraproduct $A_\natural$.
\end{theorem}

One immediate consequence of \los\, is that if $(R_n,\fm_n)$ are local rings then the ultraproduct $R_\natural$ is a quasi-local ring with unique maximal ideal $\fm_\natural$. $R_\natural$ however is rarely Noetherian; for instance, Krull's intersection theorem fails provided that the $\fm_n$ are non-nilpotent for almost all $n$ (see \cite[Item 2.4.16]{Sch10}). Hence, we define:
\begin{definition}
    Let $(R,\fm)$ be a quasi-local ring. The \emph{ideal of infinitesimals}, denoted $\mathfrak{I}_R$, is defined to be $\mathfrak{I}_R:=\bigcap\limits_{n\geq 1} \fm^n$.
\end{definition}
A construction related to the ultraproduct called the \emph{cataproduct} (and \emph{catapower}) turns out to be Noetherian in many cases.
\begin{definition}
      Let $(R_n,\fm_n,\mathscr{k}_n)$ be a collection of local rings. The \emph{cataproduct}, denoted $R_\sharp$, is defined as $R_\sharp:=R_\natural/\mathfrak{I}_{R_\natural}$. If almost all of the $R_n$ are equal to some fixed local ring $(R,\fm,\sk)$ then we refer to $R_\sharp$ as the \emph{catapower of $R$}. If $x_\natural\in R_\natural$, then we will usually denote its image in $R_\sharp$ by $\overline{x_\natural}$.
\end{definition}
If the embedding dimensions of the $R_n$ are bounded, then it is known that the cataproduct $R_\sharp$ is a Noetherian complete local ring \cite[Theorem 8.1.4]{Sch10} with maximal ideal $\fm_\natural R_\sharp$. If $(R,\fm)$ is an excellent local ring, then the catapower enjoys the property that $R\rightarrow R_\sharp$ is formally \'{e}tale \cite[Corollary 8.1.16]{Sch10}, a fact which will be used repeatedly to show that certain properties ascend from $R$ to $R_\sharp$.

We collect in the next proposition various facts about $R_\natural$ and $R_\sharp$ which may either be found in \cite{Sch10} or easily deduced from \los.
\begin{proposition}
    Let $(R,\fm)$ be a local ring, let $\{x_n\in R\}$ be a collection of elements, and let $\{I_n\subseteq R\}$ be a collection of ideals of $R$. Fix a non-principal ultrafilter $\cF$ on $\N$ and let $I_\natural=\ulim I_n$. Then:
\begin{enumerate}
    \item $I_\natural$ is a prime (resp. radical) ideal of $R_\natural$ is and only if almost all of the $I_n$ are prime (resp. radical) ideals of $R$.
    \item We may identify ultraproducts and cataproducts of quotient rings as follows:
    \reqnomode
\begin{align*}
\ulim(R/I_n)&\cong R_\natural/I_\natural, & \tag*{\cite[Item 2.1.6]{Sch10}}\\
\frac{\ulim(R/I_n)}{\mathfrak{I}_{\ulim(R/I_n)}}&\cong R_\sharp/I_\natural R_\sharp. & \tag*{\cite[Proposition 8.1.7]{Sch10}}
\end{align*}
In particular, for a fixed ideal $I\subseteq R$, the catapower of $R/I$ is $R_\sharp/I R_\sharp$ \cite[Corollary 8.1.13]{Sch10}.
\item If almost all of the $I_n$ are $\fm$-primary with $\ell_R(R/I_n)=a$ for some constant $a\in\N$, then $\ell_{R_\natural}(R_\natural/I_\natural)=a$.
\end{enumerate}
\end{proposition}

\subsection{Invariants of a (Quasi-)Local Ring} In the following sections we will need to relate the structure of our rings to elements of the extended natural numbers. We use this section to introduce certain invariants of quasi-local rings which will appear in our uniform arithmetic characterizations in further sections.
\begin{definition}\label{def: gdim}\cite[\S3]{Sch13}
    Let $(R,\fm)$ be a quasi-local ring of finite embedding dimension (i.e. the maximal ideal $\fm$ has a finite minimal generating set). The \emph{geometric dimension of $R$}, denoted $\gdim(R)$, is defined to be the smallest number of elements of $R$ generating an $\fm$-primary ideal.
\end{definition}
The following invariants are frequent inputs of the uniform arithmetic functions that we will produce.
\begin{definition}\label{def: order, degree, nilpotency degree}
 Let $(R,\fm)$ be a quasi-local ring, and let $x$ be an element of $R$.
\begin{enumerate}
    \item The \textit{($\fm$-adic) order} of $x$, denoted $\ord_R(x)$, is defined to be $\sup\{n\in \N\mid x\in \fm^n\}$.
    \item The \textit{degree} of $x$, denoted $\deg_R(x)$, is defined to be $\min\{\ell_R(R/(x,y_2,\dots,y_d))\mid y_2,\dots,y_d\in \fm\}$, where $d = \gdim(R)$.
\end{enumerate}
\end{definition}
In case the ambient ring is clear from context, we may simply write $\ord(x)$ or $\deg(x)$. Note that in the Noetherian case, Krull's intersection theorem ensures that an element has finite order if and only if it is nonzero. In the case where $R_\natural$ is obtained as the ultraproduct of Noetherian local rings of bounded embedding dimension, the condition that $\deg_{R_\natural}(x_\natural)$ be nonzero is equivalent to the condition that the image $\overline{x_\natural}\in R_\sharp$ be nonzero (i.e. that $x_\natural\not\in \mathfrak{I}_{R_\natural}$). We refer the reader to \cite[\S 3]{Sch13} for further details.

The following generalizes the notion of being part of a system of parameters to the quasi-local setting.
\begin{definition}\label{def: generic}\cite[\S 3.7]{Sch13}
Let $(R,\fm)$ be a quasi-local ring. A tuple $(x_1,\dots, x_d)$ is called a \textit{generic sequence} if it generates an $\fm$-primary ideal and if $d=\gdim(R)$. An element $x\in R$ is \textit{generic} if $x$ is part of a generic sequence.
\end{definition}
One may check that $x\in R$ is a generic element if and only if $\gdim(R/xR)=\gdim(R)-1$ \cite[Lemma 3.8]{Sch13}. Moreover, we have the following concerning ultraproducts of generic elements with bounded degree:

\begin{proposition}\cite[Corollary 5.26]{Sch13}\label{prop:generic element}
    Let $(R,\fm)$ be a local ring and let $\{x_n\}_{n\in \N}$ be a collection of ring elements such that $x_n\in R^\circ$ and such that the $\deg(x_n)$ are bounded. Then the ultraproduct $x_\natural$ is a generic element of the quasi-local ring $R_\natural$.
\end{proposition}

\section{Variants of Normality}\label{sec: variants of normality}
In this section, we characterize weak normality and seminormality via the existence of certain binary numerical functions. Our characterizations are inspired by the following theorem of Schoutens:
\begin{theorem}\cite[Theorem 12.16]{Sch13}
    Let $(R, \fm)$ be a local ring. Then $R$ is normal if and only if there exists a binary numerical function $\varphi_R$ such that for all $x, y, z\in R$, we have an inequality:
    \[\min_k\{\ord_{R/z^kR}(xy^k)\}\leq \varphi_R(\ord(x),\ord_{R/zR}(y))\]
\end{theorem}

We begin with a definition of seminormality.

\begin{definition}\label{def: seminormal}\cite[\S 3]{Swa80}
    A Noetherian ring $R$ is \textit{seminormal} if it satisfies either of the following equivalent conditions: 
\begin{enumerate}
    \item for all elements $y,z\in R$ where $y^3=z^2$, there exists some $x\in R$ such that $x^2=y$ and $x^3=z$;
    \item $R$ is reduced and for every $x\in \Tot(R)$, if $x^2,x^3\in R$ then $x\in R$.
\end{enumerate}
    
\end{definition}

\begin{theorem}\label{thm: phi characterization for seminormal}
    Let $(R,\fm)$ be an excellent local ring. $R$ is seminormal if there exists a binary numerical function $\varphi_R$ such that for every $y,z\in R$,
\begin{align*}
    \min\{\ord_{R/z^2R} (y^2), \ord_{R/z^3 R} (y^3)\}\leq \varphi_R(\ord_{R/zR}(y), \deg(z)).
\end{align*}
The converse holds if $R$ is equidimensional.
\end{theorem}
\begin{proof}
    Let $R$ be an excellent equidimensional seminormal ring and suppose there exists $(a,b)\in\N^2$ on which such a numerical function cannot be defined. Then for every $n\in \N$ there exists $y_n, z_n \in R$ where $\ord_{R/z_nR}(y_n) = a$, $\deg_{R}(z_n) = b$, $\ord_{R/z_n^2R}(y_n^2) > n$, and $\ord_{R/z_n^3R}(y_n^3) > n$. Thus, $\ord_{R_\natural/z_\natural^2R_\natural}(y_\natural^2) = \infty$, and $\ord_{R_\natural/z_\natural^3R_\natural}(y_\natural^3) = \infty$  so we may write $\overline{y_\natural}^2 = c \overline{z_\natural}^2$ and $\overline{y_\natural}^3 = d \overline{z_\natural}^3$ for some $c,d \in R_\sharp$. We also know that $\ord_{R_\natural/z_\natural R_\natural}(y_\natural) = a$. Thus, $\overline{y_\natural} \notin z_\natural R_\sharp$ and $z_\natural$ is generic by \cref{prop:generic element}. Since $R$ is reduced and equidimensional, $R_\sharp$ is equidimensional by \cite[Scholie 7.8.3(x)]{EGA} and \cite[Lemma 12.11]{Sch13}. Since $R$ is seminormal, so too is $R_\sharp$ by \cite[Corollary 8.1.16]{Sch13} and \cite[Theorem 37]{Kol16}. In particular, $R_\sharp$ is reduced, and $\overline{z_\natural}$ lies outside every minimal prime of $R_\sharp$, so $\overline{z_\natural}$ is an $R_\sharp$-regular element. Let $x = \overline{y_\natural}/\overline{z_\natural} \in \Tot(R_\sharp)$, the total quotient ring of $R_\sharp$. By the above we see that $x^2 = c \in R_\sharp$ and $x^3 = d \in R_\sharp$. Because $R_\sharp$ is seminormal, $x \in R_\sharp$ and so $x\in z_\natural R_\sharp$ which contradicts the finiteness of $\ord_{R_\natural/z_\natural R_\natural}(y_\natural)$.

    Now let us grant the existence of such a numerical function. Let $x = y/z \in \Tot(R)$ such that $x^2 = y^2/z^2 \in R$ and $x^3 = y^3/z^3 \in R$. $z$ is a nonzerodivisor so $\deg(z) \neq \infty$. We also see that $z^2x^2 = y^2$ and $z^3x^3 = y^3$, so $\ord_{R/z^2R} y^2 = \infty$ and $\ord_{R/z^3 R} y^3 = \infty$. Therefore, we know that $\ord_{R/zR}(y) = \infty$, so $y \in zR$. Thus, $x\in R$ and $R$ is seminormal. 
\end{proof}

\begin{definition}\label{def: weakly normal}
    Let $R$ be a seminormal ring, $d\in R$ a nonzerodivisor, and $b,c,e\in R$ satisfying $c^p=bd^p$ and $pc=de$ for some prime $p>0$. Then we say $R$ is \textit{weakly normal} if for all such $b$, $c$, $d$, and $e$, there exists some $a\in R$ such that $b=a^p$ and $e=pa$.
\end{definition}
Much like seminormality, weak normality may be characterized in terms of $\Tot(R)$ via \cite{Yan85}. Indeed, $R$ is weakly normal if and only if $R$ is seminormal and for every $x\in\Tot(R)$, $x\in R$ whenever $x^p, px\in R$ for some prime $p\in\Z$. We will employ both versions in proving the following theorem.

\begin{theorem}\label{thm: phi characterization for weak normality}
    Let $(R,\fm)$ be an excellent local ring. $R$ is weakly normal if there exists a ternary numerical function $\varphi_R$ such that for every $y, z\in R$, and for every prime $p$,
\begin{align*}
    \max\left\{\min\left\{\ord_{R/z^2R}(y^2),\ord_{R/z^3R}(y^3)\right\},\min\left\{\ord_{R/z^pR}(y^p), \ord_{R/zR}(py)\right\}\right\}\\ \leq \varphi_R(\ord_{R/zR}(y), \deg(z), p).
\end{align*}
The converse holds if $R$ is equidimensional.
\end{theorem}
\begin{proof}
    Let $R$ be an excellent equidimensional weakly normal ring and suppose no such numerical function existed, i.e. that there exists $(a,b,p)\in\N^3$ (where $p$ is a prime) on which such a function cannot be defined. Then for every $n$ there exists $y_n, z_n \in R$ where either
\begin{align}
    \ord_{R/z_n^2R}(y_n^2) > n, \,\ord_{R/z_n^3R}(y_n^3) > n,\, \ord_{R/z_nR}(y_n) = a, \text{ and }\deg(z_n) = b.\label{eq:wn-1}
\end{align}
or 
\begin{align}
 \ord_{R/z_n^pR}(y_n^p) > n,\, \ord_{R/z_nR}(py_n) > n,\, \ord_{R/z_nR}(y_n) = a, \text{ and }\deg(z_n) = b.\label{eq:wn-2}
\end{align} In case of (\ref{eq:wn-1}), we see that $\ord_{R_\natural/z_\natural^2 R_\natural}(y_\natural^2) = \infty$ and $\ord_{R_\natural/z_\natural^3 R_\natural}(y_\natural^3) = \infty$, and we obtain a contradiction from the proof of \Cref{thm: phi characterization for seminormal} since $R$ is seminormal. On the other hand, if (\ref{eq:wn-2}) holds, $\ord_{R_\natural/z_\natural^p R_\natural}(y_\natural^p) = \ord_{R_\natural/z_\natural R_\natural}(py_\natural) = \infty$, $\ord_{R_\natural/z_\natural R_\natural}(y_\natural) = a$, and $\deg(z_\natural) = b$. By the proof of \cref{thm: phi characterization for seminormal}, $\overline{z_\natural}$ is an $R_\sharp$-regular element, so we consider $x=\overline{y_\natural}/\overline{z_\natural}\in\Tot(R_\sharp)$. The conditions above say that $px,x^p\in R_\sharp$, which is a weakly normal ring by \cite[Theorem 37]{Kol16}. We then conclude that $x\in R_\sharp$, hence $\overline{y_\natural}\in z_\natural R_\sharp$, contradicting the finiteness of $\ord_{R_\natural/z_\natural R_\natural}(y_\natural) = a$.
    
Suppose such a numerical function exists. By the proof of \cref{thm: phi characterization for seminormal}, we see that $R$ is seminormal. Now let $z\in R$ be a nonzero divisor, $x = y/z \in \Tot(R)$ such that $x^p = y^p/z^p \in R$, and $px = py/z \in R$ for some prime $p$. Then, $\min\{\ord_{R/z^pR}(y^p), \ord_{R/zR}(py)\} = \infty$. Since $\deg(z)$ is finite, $\ord_{R/zR}(y) = \infty$, hence $x\in R$.
\end{proof}

\section{\texorpdfstring{$F$}{F}-singularities}\label{sec: F-singularities}
We briefly review the notions of tight closure and Frobenius closure. For the entirety of this section, $(R,\fm)$ will be a $d$-dimensional local ring of prime characteristic $p>0$. We denote restriction of scalars along the ($e$th iterated) Frobenius endomorphism $F^e:R\rightarrow R$ via $F^e_* R$. That is, $F^e_*R$ is identified with $R$ as an abelian group, but the $R$-module structure is given by $r\cdot F^e_*(s)=F^e_*(r^{p^e} s)$ for all $r,s\in R$. For an ideal $I\subseteq R$, we denote $I^{[p^e]}$ the ideal $\langle r^{p^e}\mid r\in I\rangle$. If $I=(x_1,\dots, x_n)$ then we identify $I^{[p^e]}=(x_1^{p^e},\dots, x_n^{p^e})$ which does not depend on the choice of generating set.

Given an ideal $I\subseteq R$, the \textit{Frobenius closure} $I^F$ (resp. \textit{tight closure} $I^*$) of $I$ is given by
\begin{align*}
    I^F:=&\{r\in R\mid r^{p^e}\in I^{[p^e]}\text{ for all } e\gg 0\},\\
    I^*:=& \{r\in R\mid \exists c\in R^\circ\text{ such that } cr^{p^e}\in I^{[p^e]}\text{ for all } e\gg 0\}.
\end{align*}
If $R$ is reduced, or if $R$ is arbitrary and $\Height(I)>0$, the condition that $r\in I^*$ may be restated to require the existence of $c\in R^\circ$ such that $cr^{p^e}\in I^{[p^e]}$ for \emph{all} $e>0$ (see \cite[Theorem 1.3(b)]{Hun96} or \cite[Proposition 4.1(c)]{HH90}).

We will also require the module-theoretic generalizations of Frobenius and tight closure. Fix a set of generators $\fm=(r_1,\dots, r_n)$ and consider the \v{C}ech complex

\begin{equation*}
    \check{C}^\bullet(r_1,\dots, r_n): 0\to R\to \bigoplus\limits_{i=1}^n R_{r_i}\to \bigoplus\limits_{i>j} R_{r_j r_i}\to \cdots\to R_{r_1\cdots r_n}\to 0.
\end{equation*}
For each $0\leq i\leq d$ we compute the $i$th local cohomology module $H^i_\fm(R)$ via the $i$th cohomology $H^i(\check{C}^\bullet(r_1,\dots, r_n))$, which is independent up to isomorphism of choice of generators for $\fm$. For each $e$ the Frobenius map $R\rightarrow F^e_* R$ induces a Frobenius action on each $H^i_\fm(R)$ via 
\begin{align*}
    F^e: H^i_\fm(R)\rightarrow H^i_\fm(F^e_* R).
\end{align*}
In case $i=d$, this map can be made explicit via the $R$-module isomorphism $H^d_\fm(F^e_* R)\cong F^e_* H^d_\fm(R)$. In fact, if $\eta=\left[\frac{s}{r_1^t\cdots r_d^t}\right]\in H^d_\fm(R)$ is a \v{C}ech class then the Frobenius action (after identifying $H^d_\fm(F^e_* R)\cong H^d_\fm(R)$ as abelian groups) is simply $F(\eta)=\left[\frac{s^p}{r_1^{tp}\cdots r_d^{tp}}\right]$. We then define
\begin{align*}
    0^F_{H^d_\fm(R)} := & \{\eta\in H^d_\fm(R)\mid \eta\in\ker(H^d_\fm(R)\to F^e_* H^d_\fm(R)) \text{ for all } e\gg 0\},\\
    0^*_{H^d_\fm(R)} := & \left\{\eta\in H^d_\fm(R)\mathrel{\Big|} \substack{\exists c\in R^\circ\text{ such that } \\ \eta\in \ker\left(H^d_\fm(R)\to F^e_* H^d_\fm(R)\stackrel{\cdot F^e_* c}{\to} F^e_*H^d_\fm(R)\right)\text{ for all } e\gg0}\right\},
\end{align*}
the \textit{Frobenius closure} (resp. \textit{tight closure}) of zero in $H^d_\fm(R)$. Note that we always have $I^F\subseteq I^*$ and $0^F_{H^d_\fm(R)}\subseteq 0^*_{H^d_\fm(R)}$, and understanding when these containments are equalities motivates the notion of $F$-nilpotence.
\begin{definition}
    A local ring $(R,\fm)$ of prime characteristic $p>0$ is said to be \emph{$F$-nilpotent} if $0^F_{H^d_\fm(R)}=0^*_{H^d_\fm(R)}$ and the map $F^e:H^i_\fm(R)\rightarrow H^i_\fm(F^e_* R)$ is the zero map for all $0\leq i<d$ and all $e\gg 0$.
\end{definition}
We will employ the following ideal-theoretic characterization of $F$-nilpotence due to Polstra and Pham.
\begin{theorem}\cite[Theorem A]{PQ19}
    An excellent equidimensional local ring $(R,\fm)$ of prime characteristic $p>0$ is $F$-nilpotent if and only if $I^*=I^F$ for all ideals $I\subseteq R$ generated by a system of parameters.
\end{theorem}

In our treatment of $F$-nilpotent rings, we will need the following lemma which says that the Frobenius closure of an ultraproduct is contained in the ultraproduct of Frobenius closures, at least when the ideals have bounded number of generators.

\lemma\label{lemma: frobenius closure subset} Let $R$ be a Noetherian ring of prime characteristic $p>0$, and let $\{I_n\}_{n\in\N}$ be a collection of ideals of $R$ each generated by $N$ elements. Let $\cF$ be a non-principal ultrafilter on $\N$. Then $ (\ulim_{n \rightarrow \cF} I_n)^F\subseteq \ulim_{n \rightarrow \cF}(I_n^F)$ as ideals of $R_\natural$.
\begin{proof}
For each $n\in \N$ fix a generating set $I_n=(b_{1,n},\dots, b_{N,n})$. Let $\ulim_{n\to \cF} r_n\in (\ulim_{n \rightarrow \cF} I_n)^F$. Then there exists $e_0$ such that for every $e\geq e_0$, $$(\ulim r_n)^{p^e}=\ulim (r_n^{p^e})\in \left(\ulim_{n\to\cF} I_n\right)^{[p^e]}.$$ We note that $\left(\ulim_{n\to\cF} I_n\right)^{[p^e]}$ is finitely generated by \cite[Item 2.1.6]{Sch10} (with generators $b_{i\natural}^{p^e}$), so we may write
\begin{align*}
    \ulim(r_n^{p^e})=\sum\limits_{t=1}^N(\ulim_n a_{t,n}) (\ulim_n b_{t,n}^{p^e})=\ulim_n\left(\sum\limits_{t=1}^N a_{t,n} b_{t,n}^{p^e}\right)
\end{align*}
for some $a_{t,n}\in R$. This is equivalent to $\left\{n\in\N\mid r_n^{p^e}=\sum\limits_{t=1}^N a_{t,n} b_{t,n}^{p^e}\right\}\in\cF$. It follows that $\{n\in\N\mid r_n\in I_n^{[p^e]}\}\in\cF$ for every $e\geq e_0$, hence $\{n\in\N\mid r_n\in I_n^F\}\in \cF$. This implies that $\ulim r_n\in\ulim_{n\to\cF} (I_n^F)$ as desired.
\end{proof}

\begin{theorem}\label{thm: F-nilpotent numerical function}
    Let $(R,\fm,\sk)$ be an excellent local ring of prime characteristic $p>0$. If $R$ is $F$-nilpotent then there exists a ternary numerical function $\varphi_R: (\N \cup \{\infty \})^3 \rightarrow \N \cup \{\infty \}$ such that for all $x,y \in R$ and every parameter ideal of $I \subseteq R$, 
    \[ \min_e \{\ord_{R/I^{[p^e]}}(xy^{p^e})\} \leq \varphi_R(\deg(x), \ell_{R}(R/I^F), \ord_{R/I^F}(y)). \]
    The converse holds if $R$ is equidimensional. 
\end{theorem}
\begin{proof}
    Let $R$ be $F$-nilpotent. Suppose for the sake of contradiction that there exists no ternary numerical function such that the inequality holds. Then we can find $a,b,c \in \N$ such that for every $n\in N$, there exists $x_n, y_n \in R$ and parameter ideal $I_n \subseteq R$ such that $\deg(x_n) = a$, $\ell_{R}(R/I_n^F) = b$, $\ord_{R/I_n^F}(y_n) = c$, and $\ord_{R/I_n^{[p^e]}}(x_ny_n^{p^e}) >n$ for all $e$. Note that since $c$ is finite,
\begin{align*}
    \overline{y_\natural} \notin \ulim\limits_{n\to\cF}(I_n^F)R_\sharp\stackrel{(\ref{lemma: frobenius closure subset})}{\supseteq}(\ulim\limits_{n\to\cF}I_n)^F R_\sharp=(I_\natural)^F R_\sharp\supseteq (I_\natural R_\sharp)^F
\end{align*}
so that $\overline{y_\natural} \notin (I_\natural R_\sharp)^F$. The residue field of $R_\sharp$ is $\sk_\natural$ and $\sk\rightarrow \sk_\natural$ is separable by \cite[Corollary 8.1.16]{Sch13}. We then observe that the excellent rings $R$ and $R_\sharp$ have a common test element by \cite[Theorem 6.21]{HH94}, so we may apply \cite[Theorem 4.4(2)]{KMPS23} to conclude that $R_\sharp$ is $F$-nilpotent. Note that $I_\natural R_\sharp$ is a parameter ideal of $R_\sharp$, hence $\overline{y_\natural}\notin (I_\natural R_\sharp)^*$ by \cite[Theorem A]{PQ19}. Because $\ord_{R/I^{[p^e]}}(x_ny_n^{p^e}) > n$ for all $e$, 
\begin{align*}
    \overline{x_\natural}\cdot \overline{y_\natural}^{p^e}=\overline{x_\natural y_\natural^{p^e}} \in I_\natural^{[p^e]} R_\sharp  
\end{align*}
 for all $e$. Note that $x_\natural\in R_\natural$ is generic by \cref{prop:generic element}, so the image $\overline{x_\natural}$ is not contained in any minimal prime of $R_\sharp$ of maximal dimension. Since $R_\sharp$ is $F$-nilpotent, it is equidimensional by \cite[Proposition 2.8]{PQ19} so that $\overline{x_\natural}\in (R_\sharp)^\circ$. Thus, $\overline{y_\natural} \in (I_\natural R_\sharp)^*$, a contradiction.

Let there exist a ternary numerical function $\varphi_R$ such that the inequality holds. Let $y \in I^*$, where $I$ is a parameter ideal. Since $\Height(I)>0$, there exists $x\in R^\circ$ such that $xy^{p^e} \in I^{[p^e]}$ for all $e>0$ by \cite[Proposition 4.1(c)]{HH90}. So, $\ord_{R/I^{[p^e]}}(xy^{p^e}) = \infty$ for all $e$. Note that $\deg(x)<\infty$ since $x\in R^\circ$ and moreover $\ell_{R}(R/I^F) \leq \ell_{R}(R/I)<\infty$. Since
\begin{align*}
    \infty=\min_e \{\ord_{R/I^{[p^e]}}(xy^{p^e})\} \leq \varphi_R(\deg(x), \ell_{R}(R/I^F), \ord_{R/I^F}(y))
\end{align*}
    we know that the argument $\ord_{R/I^F}(y)$ must be infinite. Therefore, $y \in I^F$ and $I^F = I^*$. Since $R$ is excellent and equidimensional, it follows from \cite[Theorem A]{PQ19} that $R$ is $F$-nilpotent, as desired.  
\end{proof}

If $I\subseteq R$ is any ideal, Noetherianity of $R$ tells us that there exists $e_0\gg 0$ so that $(I^F)^{[p^e]}=I^{[p^e]}$ for every $e\geq e_0$. We call the smallest such $e_0$ the \textit{Frobenius test exponent of} $\mathit{I}$, denoted $\fte(I)$. The \textit{Frobenius test exponent of} $\mathit{R}$ is defined via $$\fte(R):=\min\{e_0\mid (I^F)^{[p^e]}=I^{[p^e]}\text{ for all $e\geq e_0$ and all parameter ideals }I\}$$ if this minimum exists, and $\fte(R)=\infty$ otherwise. We will employ the following result of Pham regarding the existence of this uniform bound in a special case.
\begin{theorem}\cite[Main Theorem]{Quy19}\label{thm:fte exists}
    If $(R,\fm)$ is a generalized Cohen--Macaulay local ring of prime characteristic $p>0$, then $\fte(R)<\infty$.
\end{theorem}

Since \cref{thm:fte exists} will only be applied in the Cohen--Macaulay setting, we postpone the definition of a generalized Cohen--Macaulay ring to \cref{sec: Variants of CM}. We next discuss a uniform arithmetic characterization of $F$-injectivity. A local ring $(R,\fm)$ of prime characteristic $p>0$ is said to be $\mathit{F}$-\textit{injective} if the Frobenius maps $F:H^i_\fm(R)\to H^i_\fm(F_* R)$ on local cohomology are injective for all $0\leq i\leq \dim(R)$. We will use a well-known elementary characterization of this notion in terms of Frobenius closure being a trivial operation on parameter ideals (at least when $R$ is Cohen--Macaulay).

\begin{theorem}\label{thm: F-injective numerical function}
    Let $(R,\fm,\sk)$ be a Cohen--Macaulay local ring of prime characteristic $p>0$. $R$ is $F$-injective if there exists a binary numerical function $\varphi_R: (\N \cup \{\infty \})^2 \rightarrow \N \cup \{\infty \}$ such that for all $x \in R$ and every parameter ideal of $I \subseteq R$, 
    \[ \min_{e\geq \fte(I)} \{\ord_{R/I^{[p^e]}}(x^{p^e})\} \leq \varphi_R(\ell_{R}(R/I), \ord_{R/I}(x)). \]
\end{theorem}
\begin{proof}
    Let $R$ be $F$-injective. Suppose towards a contradiction that no such binary numerical function exists for which the stated inequality holds. Then, there exists $a,b \in \N$ such that for every $n\in \N$ we can find $x_n \in R$ and parameter ideal $I_n \subseteq R$ such that $\ell_{R}(R/I_n) = a$, $\ord_{R/I_n}(x_n) = b$, and $\ord_{R/I_n^{[p^e]}}(x_n^{p^e}) > n$ for all $e\geq \fte(R)$ (by \cref{thm:fte exists}). Since $b$ is finite, we see that $\overline{x_\natural} \notin I_\natural R_\sharp$, a parameter ideal of $R_\sharp$. Again using that $\sk\to \sk_\natural$ is separable by \cite[Corollary 8.1.16]{Sch13}, we see (for example by \cite[Theorem A]{DM19}) that $R_\sharp$ is $F$-injective. Because $\ord_{R/I^{[p^e]}}(x_n^{p^e}) > n$ for all $e\geq \fte(R)$, we see that $\overline{x_\natural}^{p^e} \in I_\natural^{[p^e]} R_\sharp$ for all $e\geq \fte(R)$ and hence $\overline{x_\natural} \in (I_\natural R_\sharp)^F$. Since $R$ (and hence $R_\sharp$) is Cohen--Macaulay, we observe that $\overline{x_\natural}\in I_\natural R_\sharp$ by \cite[Corollary 3.9]{QS17}, a contradiction.

    Suppose such a binary numerical function $\varphi_R$ exists for which the inequality holds. Let $I\subseteq R$ be a parameter ideal and let $x \in I^F$. Then $x^{p^e} \in I^{[p^e]}$ for all $e\geq \fte(I)$ so that $\ord_{R/I^{[p^e]}}(x^{p^e}) = \infty$ for all $e\geq \fte(I)$. Because $\ell_{R}(R/I)<\infty$ and 
    \[ \infty = \min_{e\geq \fte(I)} \{\ord_{R/I^{[p^e]}}(x^{p^e})\} \leq \varphi_R(\ell_{R}(R/I), \ord_{R/I}(x)), \]
    we know that $\ord_{R/I}(x)$ must be infinite. Therefore, $x \in I$ and $I^F = I$. Using \cite[Theorem 3.7]{QS17} we see that $R$ is $F$-injective, as claimed.
\end{proof}

We conclude this section with a characterization of $F$-purity. Recall the following definition:
\begin{definition}
Let $(R,\fm)$ be a local ring of prime characteristic $p>0$. $R$ is \emph{$F$-pure} if for every $R$-module $M$, the map $M\to M\otimes_R F^e_* R$ is injective. $R$ is \emph{cyclically $F$-pure} if $I=I^F$ for every ideal $I\subseteq R$.
\end{definition}

$F$-pure rings are always cyclically $F$-pure, and the notions coincide, for example, for rings which are excellent and reduced \cite{Hoc77}.
\begin{theorem}\label{thm: F-pure phi function}
    Let $(R,\fm)$ be a reduced excellent local ring of prime characteristic $p>0$. $R$ is $F$-pure if and only if there exists a ternary numerical function $\varphi_R: (\N \cup \{\infty \})^3 \rightarrow \N \cup \{\infty \}$ such that for all $x \in R$ and every $\fm$-primary ideal of $I \subseteq R$, 
    \[ \min_{e\geq \fte(I)} \{\ord_{R/I^{[p^e]}}(x^{p^e})\} \leq \varphi_R(\ord_{R/I}(x), \fte(I),\ell_R(R/I)). \]
\end{theorem}
\begin{proof}
    Let $R$ be $F$-pure and suppose there exists $(a,b,c) \in \N^3$ on which such a numerical function $\varphi_R$ cannot be defined. Then for every $n$, we may find an element $x_n \in R$ and an $\fm$-primary ideal $I_n \subseteq R$ such that 
\begin{align*}
\ord_{R/I_n}(x_n) = a,\\
\fte(I_n)=b,\\
\ell_R(R/I_n)=c
\end{align*}
and $\ord_{R/I_n^{[p^e]}}(x_n^{p^e}) > n$ for all $e \geq b$. By \los, 
\begin{align*}
    \ord_{R_\natural/I_\natural}(x_\natural) = a,\\
    \fte(I_\natural R_\sharp)=b,\\
    \ell_{R_\natural}(R_\natural/I_\natural)=c,\\
    \text{and }x_\natural^{p^e}\in\mathfrak{I}_{R_\natural/I_\natural^{[p^e]}}
\end{align*}
for all $e \geq b$. Thus, $\overline{x_\natural}^{p^e} \in I_\natural^{[p^e]}R_\sharp$ for all $e \geq c$, so $\overline{x_\natural}\in (I_\natural R_\sharp)^F\setminus I_\natural R_\sharp$. However, $R_\sharp$ is $F$-pure (for example, by \cite[Theorem 7.4]{MP21}) so every ideal is Frobenius closed, a contradiction.\newline

Suppose that such a numerical function exists. Since $R$ is approximately Gorenstein, it suffices to check that $R$ is cyclically $F$-pure. By Krull's intersection theorem and the fact that arbitrary intersections of Frobenius closed ideals are Frobenius closed, it suffices to check that $I=I^F$ for all $\fm$-primary ideals $I$. Fix such an ideal and let $x \in I^F$. Then, $x^{p^e} \in I^{[p^e]}$ for all $e \geq \fte(I)$, and $\min_{e\geq \fte(I)} \{\ord_{R/I^{[p^e]}}(x^{p^e})\} = \infty$. Then, $\ord_{R/I}(x) = \infty$, so $x \in I$. Therefore, $I^F = I$ and $R$ is $F$-pure. 
\end{proof}

\section{Variants of the Cohen--Macaulay property}\label{sec: Variants of CM}
This section was inspired by the following characterization due to Schoutens of the Cohen--Macaulay property:
\begin{theorem}\cite[Theorem 12.14]{Sch13}
    Let $R$ be a local ring with dimension $d$. Then $R$ is Cohen-Macaulay if and only if there exists a binary numerical function $\varphi_R$ such that for all $d$-tuples $\mathbf x := (x_1, \ldots, x_d)$ and $(y_1,\ldots, y_d)$ of elements of $R$ with $\mathbf x$ a system of parameters, we have:
    \[\ord_R(x_1y_1+\ldots + x_dy_d) \leq \varphi_R(\ell_R(R/\mathbf xR), \ord_{R/(x_1,\ldots x_{d-1})R}(y_d)).\]
\end{theorem}
In this section, we explore analogous phenomena for two properties which are weaker than being Cohen--Macaulay --- \emph{Buchsbaum} and \emph{generalized Cohen--Macaulay}. In general, we have
\begin{align*}
\text{Cohen--Macaulay} \Rightarrow \text{Buchsbaum} \Rightarrow \text{generalized Cohen--Macaulay}
\end{align*}
and neither arrow is reversible. We recall the following defintions:

\begin{definition} A $d$-dimensional local ring $(R,\fm,\sk)$ is said to be \textit{Buchsbaum} if it satisfies any of the following equivalent conditions:
\begin{enumerate}
    \item The quantity $\ell_R(R/I)-e(I)$ is independent of choice of parameter ideal $I\subseteq R$ (here, $e(I)$ is the multiplicity of $I$);
    \item \cite[I, Theorem 1.12]{SV86} For every system of parameters $x_1,\dots, x_d\in R$ and every $1\leq i\leq d$, we have the equality of ideals $(x_1,\dots, x_{i-1}):x_i = (x_1,\dots, x_{i-1}):\fm$;\
    \item \cite[Theorem 2.3]{Sch82} The truncation $\tau^{<d}\mathbf{R}\Gamma_\fm R$ of the right derived $\fm$-power torsion of $R$ is quasi-isomorphic (as an object of the derived category $D(R)$) to a complex of $\sk$-vector spaces.\label{def:buchsbaum-4}
\end{enumerate}
\end{definition}

\begin{definition}
    A $d$-dimensional local ring $(R,\fm)$ is said to be \emph{generalized Cohen--Macaulay} if it satisfies any of the following equivalent conditions:
    \begin{enumerate}
        \item $\ell_R(H^i_\fm(R))<\infty$ for all $i<d$;
        \item \cite[Theorem 37.10]{HIO88} There exists an $\fm$-primary ideal $I\subseteq R$ such that for every system of parameters $x_1,\dots, x_d\in I$, we have an equality of ideals $(x_1,\dots, x_{d-1}):x_d=(x_1,\dots, x_{d-1}):I$;
        \item \cite[Satz 3.3]{STC78} There exists an integer $n\in\N$ such that for every system of parameters $x_1,\dots, x_d\in R$ and every $1\leq i\leq d$, we have the containment $(x_1,\dots, x_{i-1}):x_i\subseteq (x_1,\dots, x_{i-1}):\fm^n$;
        \item \cite[Lemma 1.5]{LT06} There exists an integer $n\in\N$ such that for every system of parameters $x_1,\dots, x_d\in R$, we have the containment $(x_1,\dots, x_{d-1}):x_d\subseteq (x_1,\dots, x_{d-1}):\fm^n$.
    \end{enumerate}
\end{definition}

We first give a characterization of the Buchsbaum property. 

\begin{theorem}\label{thm: phi characterization of Buchsbaum}
    Let $(R,\fm,\sk)$ be a $d$-dimensional local ring. Then $R$ is Buchsbaum if and only if there exists a binary numerical function $\varphi_R: (\N \cup \{\infty \})^2 \rightarrow \N \cup \{\infty \}$ such that for any system of parameters $\mathbf{x}=x_1,\dots, x_d\in R$ and any $d$-tuple $\mathbf{y}=y_1,\dots, y_d\in R$,
\begin{align*}\ord_R(x_1y_1+\cdots+x_dy_d)\leq\varphi_R(\ell_R(R/\mathbf{x}R),\ord_{R/(x_1,\dots, x_{d-1}):\fm}(y_d)).
    \end{align*}
\end{theorem}
\begin{proof}
    Suppose that there exists such a function and let $z_1,\dots, z_d\in R$ be a system of parameters. Fix $1\leq i\leq d$ and let $a_i\in(z_1,\dots, z_{i-1}):z_i$. Write $a_1z_1+\cdots+a_{i}z_{i}=0$. Now for each $t\geq 1$ and $1\leq j\leq d$ define
    \begin{align*}
        x_{t,j}:=\begin{cases}z_{i+j}^t:& 1\leq j\leq d-i\\ z_{i+j-d}:& d-i+1\leq j\leq d\end{cases}, \hspace{.5cm} y_{t,j}:=\begin{cases}0:& 1\leq j\leq d-i\\ a_{i+j-d}:& d-i+1\leq j\leq d\end{cases}.
    \end{align*}
Note that $\mathbf{x}_t$ is still a system of parameters, and we have $x_{t,1}y_{t,1}+\cdots+x_{t,d} y_{t,d}=a_1z_1+\cdots +a_iz_i=0$. Since $\ell_R(R/\mathbf{x}_t R)<\infty$, it follows that $\ord_{R/(x_{t,1},\dots, x_{t,d-1}):\fm}(y_{t,d})=\infty$ for all $t$. Since $a_i=y_{t,d}$ for all $t$, we see that
\begin{align*}
    a_i\in (x_{t,1},\dots, x_{t,d-1}):\fm = ((z_1,\dots, z_{i-1})+(z_{i+1}^t,\dots, z_d^t)):\fm
\end{align*}
for all $t$. Hence, 
\begin{align*}
a_i&\in\bigcap\limits_{t>0}\left(\left((z_1,\dots, z_{i-1})+\left(z_{i+1}^t,\dots, z_d^t\right)\right):\fm\right)\\
&=\left(\bigcap\limits_{t>0}\left((z_1,\dots, z_{i-1})+(z_{i+1}^t,\dots, z_d^t)\right)\right):\fm\\\
&=(z_1,\dots, z_{i-1}):\fm
\end{align*}
by the Krull intersection theorem, so $R$ is Buchsbaum as claimed.

Now suppose that $R$ is Buchsbaum and that there exists no such numerical function. That is, we can find $(a,b)\in\N^2$ such that for every $n>0$, there exists a system of parameters $x_{1,n},\dots, x_{d,n}$ and a $d$-tuple $y_{1,n},\dots, y_{d,n}$ for which
\begin{align*}
    \ell_R(R/\mathbf{x_n}R)=a, \\
    \ord_{R/(x_{1,n},\dots, x_{d-1,n}):\fm}(y_{d,n})=b,\\
    \ord_R(x_{1,n}y_{1,n}+\cdots+x_{d,n}y_{d,n})\geq n.
\end{align*}
By \los, we know that 
\begin{align}
    \ell_{R_\natural}(R_\natural/\mathbf{x}_\natural R_\natural)&=a,\label{eq:buchsbaum-1}\\
    \ord_{R_\natural/(x_{1\natural},\dots, x_{{d-1}\natural}):_{R_\natural}\fm_\natural}(y_{d\natural})&=b,\label{eq:buchsbaum-2}\\
x_{1\natural}y_{1\natural}+\cdots+x_{d\natural}y_{d\natural}&\in\mathfrak{I}_{R_\sharp}.\label{eq:buchsbaum-3}
\end{align}
The containment (\ref{eq:buchsbaum-3}) implies that
\begin{align}
\overline{x_{1\natural}y_{1\natural}+\cdots+x_{d\natural}y_{d\natural}}=0\in R_\sharp\label{eq:buchsbaum-4}.
\end{align}
 We note that $\overline{\mathbf{x}_\natural}$ is a system of parameters for $R_\sharp$. The residue field extension $\sk\rightarrow \sk_\natural$ induced by the flat local ring map $R\rightarrow R_\sharp$ is separable, and $\tau^{<d}\mathbf{R}\Gamma_{\fm_\sharp} R_\sharp\cong  (\tau^{<d}\mathbf{R}\Gamma_\fm R)\otimes_R R_\sharp$. Since $R$ is Buchsbaum, $\tau^{<d}\mathbf{R}\Gamma_\fm R$ is quasi-isomorphic in $D(R)$ to a complex of $\sk$-vector spaces. It follows that $\tau^{<d}\mathbf{R}\Gamma_{\fm_\sharp} R_\sharp$ is quasi-isomorphic in $D(R_\sharp)$ to a complex of $\sk_\natural$-vector spaces, hence $R_\sharp$ is Buchsbaum too. Now observe that
 \begin{align*}
     \overline{y_{d\natural}}&\in(x_{1\natural},\dots, x_{{d-1}\natural})R_\sharp:_{R_\sharp} x_{d\natural}R_\sharp\\
     &=(x_{1\natural},\dots, x_{{d-1}\natural})R_\sharp:_{R_\sharp} \fm_\sharp
 \end{align*}
 which contradicts \cref{eq:buchsbaum-2}.
\end{proof}

We now prove a similar characterization for the generalized Cohen--Macaulay property. A distinguishing feature is that this characterization comes with a \emph{pair} $(\varphi_R,N)$ for some integer $N$, rather than merely a numerical function as in the previous work.
\begin{theorem}\label{thm: phi characterization of generalized Cohen-Macaulay}
    Let $(R,\fm)$ be a local ring with $\dim(R)=d\geq 1$. $R$ is generalized Cohen--Macaulay if and only if there exists an integer $N>0$ and a binary numerical function $\varphi_R: (\N \cup \{\infty \})^2 \rightarrow \N \cup \{\infty \}$ such that for all systems of parameters $x_1,\dots, x_d\in R$, and all elements $y_1, \dots, y_d \in R$, 
    \begin{align*} 
    \ord_R(x_1y_1 + \dots + x_dy_d) \leq \varphi_R(\ell_{R}(R/(x_1, \dots, x_d)R), \ord_{R/(x_1, \dots, x_{d-1}) : \fm^N}(y_d)).
    \end{align*}
\end{theorem}

\begin{proof}
    Let $R$ be generalized Cohen--Macaulay. Towards a contradiction, suppose that for every $N>0$ there exists $(a_N,b_N)\in\N^2$ with the property that for every $n\in \N$, there exists a system of parameters $x_{1,N,n},\dots, x_{d,N,n}\in\fm$ and $y_{1,N,n},\dots, y_{d,N,n}\in R$ such that
\begin{align}
    \ell_R(R/(x_{1,N,n},\dots, x_{d,N,n})R)=a_N\\
    \ord_{R/(x_{1,N,n},\dots, x_{d-1,N,n}):_R \fm^N} (y_{d,N,n}) = b_N,\\
    \text{ and }\ord_R(x_{1,N,n}y_{1,N,n}+\cdots+x_{d,N,n}y_{d,N,n})>n.
\end{align}    
By \los\, for every $N$ we know that
\begin{align*}
\ord_{R_\natural/(x_{1,N\natural}, \dots, x_{d-1,N\natural}):\fm_\natural^N}(y_{d,\natural}) = b_N,\\
\text{and }\ord_{R_\natural}(x_{1,N\natural}y_{1,N\natural} + \dots + x_{d,N\natural}y_{d,N\natural}) = \infty.
\end{align*}
Thus, $\overline{x_{1,N\natural}}\overline{y_{1,N\natural}} + \dots + \overline{x_{d,N\natural}}\overline{y_{d,N\natural}} = 0$ and $\overline{y_{d,N\natural}} \not\in (\overline{x_{1,N\natural}}, \dots,\overline{x_{d-1,N\natural}}):_{R_\sharp} \fm_\sharp^N$ for all $N$. We claim that $R_\sharp$ is generalized Cohen--Macaulay. To see this, note that for $i<\dim(R)=\dim(R_\sharp)$,
    \begin{align*}
        \ell_{R_\sharp}(H^i_{\fm_\sharp}(R_\sharp))=\ell_{R_\sharp}(H^i_\fm(R)\otimes_R R_\sharp)=\ell_R(H^i_\fm(R))\ell_{R_\sharp}(R_\sharp/\fm R_\sharp)<\infty.
    \end{align*}
    It follows that there exists $A>0$ so that for every $N$,
\begin{align*}
    (\overline{x_{1,N\natural}},\dots, \overline{x_{d-1,N\natural}}):_{R_\sharp} \overline{x_{d,N\natural}}\subseteq (\overline{x_{1,N\natural}},\dots, \overline{x_{d-1,N\natural}}):_{R_\sharp} \fm_\sharp^A.
\end{align*}
Taking $N>A$ gives a contradiction from
\begin{align*}
    \overline{y_{d,N\natural}}\in (\overline{x_{1,N\natural}},\dots, \overline{x_{d-1,N\natural}}):_{R_\sharp} \overline{x_{d,N\natural}}&\subseteq (\overline{x_{1,N\natural}},\dots, \overline{x_{d-1,N\natural}}):_{R_\sharp} \fm_\sharp^A\\
    &\subseteq (\overline{x_{1,N\natural}},\dots, \overline{x_{d-1,N\natural}}):_{R_\sharp} \fm_\sharp^N.
\end{align*}
 \newline
Let such an $N>0$ and such a numerical function exist and let $x_1,\dots, x_d$ be a system of parameters. Suppose that $y_d\in (x_1,\dots, x_{d-1}):x_d$ and write $x_1y_1+\cdots+x_dy_d=0$ for some $y_1,\dots, y_{d-1}\in R$. It follows that $\ord_{R/(x_1, \dots, x_{d-1}) : \fm^N}(y_d)=\infty$ which shows that $(x_1,\dots, x_{d-1}):x_d\subseteq (x_1,\dots, x_{d-1}):\fm^N$ as desired.
\end{proof}

\section{Questions}
We conclude the article with some ideas for further investigation. The first is a restatement of a sentiment appearing in \cite{Sch13}.
\begin{question}
    Can one find proofs of any of the results in the present paper or in \cite[\S 12]{Sch13} that do not use ultraproducts? If so, how severe are the uniform bounds? Are they ever linear?
\end{question}

The bound in the uniform characterization of the Cohen--Macaulay property that Schoutens gives remarkably does not depend on the ring. Rather, using a package of results involving the depth and multiplicity of an ultraproduct of local rings (of bounded embedding dimension), it is shown that the stated bound applies to all rings of a given multiplicity and Krull dimension. See \cite[Theorem 12.14]{Sch13} and \S 8 of \emph{op. cit.} for more details. Motivated by this, we propose the following:

\begin{question}
    Can \cref{thm: F-nilpotent numerical function,thm: phi characterization of Buchsbaum,thm: phi characterization of generalized Cohen-Macaulay} be improved to give uniform bounds for all $F$-nilpotent (resp. Buchsbaum, generalized Cohen--Macaulay) rings of a set of fixed parameters (such as $\dim(R)$, $e(R)$, $\text{F-depth}(R)$, $\min\{\ell_R(R/I)-e(I)\mid I\subseteq R\text{ parameter ideal}\}$)?
\end{question}

\newpage
\thispagestyle{plain}
\begin{landscape}
    \begin{ThreePartTable}
 \begingroup
      \tiny
      \renewcommand{\arraystretch}{1.3}
      \begin{longtable}[c]{*{6}{c}}
      \caption{\normalsize Summary of uniform arithmetic characterizations for local properties $\cP$.\\ A local ring $(R,\fm)$ satisfies property $\cP$ if and only if there exists a function\\ $\varphi_R$ with the stated behavior, provided the relevant assumptions in the footnote.}\label{table: uniform arithmetic characterizations}\\
    \toprule
        $\cP$ & Input & Numerical function &  Reference\\
        \cmidrule(lr){1-1} \cmidrule(lr){2-2} \cmidrule(lr){3-4}
        Weakly $F$-regular\tnotex{t:exc} \tnote{,} \tnotex{t:omit} & $\forall \sqrt{I}=\fm,\forall x,y\in R$ &$\min_e\{\ord_{R/I^{[p^e]}}(xy^{p^e})\}\leq\varphi_R(\deg(x),\ell_R(R/I),\ord_{R/I}(y))$ & \cite[Thm. 12.18]{Sch13}\\\hdashline
        $F$-pure\tnotex{t:exc} \tnote{,} \tnotex{t:red} & $\forall \sqrt{I}=\fm,\forall x\in R$& $\min_{e\geq\fte(I)} \{\ord_{R/I^{[p^e]}}(x^{p^e})\} \leq \varphi_R(\ell_{R}(R/I), \fte(I),\ord_{R/I}(x))$ & \cref{thm: F-pure phi function}\\\hdashline
        $F$-rational\tnotex{t:exc} & $\forall I$ parameter ideal, $\forall x,y\in R$ &  $\min_e\{\ord_{R/I^{[p^e]}}(xy^{p^e})\}\leq\varphi_R(\deg(x),\ell_R(R/I),\ord_{R/I}(y))$ & \cite[Thm. 12.14]{Sch13}\\\hdashline
        $F$-injective\tnotex{t:cm} & $\forall I$ parameter ideal, $\forall x\in R$ &$\min_{e\geq \fte(I)} \{\ord_{R/I^{[p^e]}}(x^{p^e})\} \leq \varphi_R(\ell_{R}(R/I), \ord_{R/I}(x))$& \cref{thm: F-injective numerical function}\\\hdashline
        $F$-nilpotent\tnotex{t:exc} \tnote{,} \tnotex{t:equi} & $\forall I$ parameter ideal, $\forall x,y\in R$ & $\min_e \{\ord_{R/I^{[p^e]}}(xy^{p^e})\} \leq \varphi_R(\deg(x), \ell_{R}(R/I^F), \ord_{R/I^F}(y))$ & \cref{thm: F-nilpotent numerical function}\\
        \hline 
        normal\tnotex{t:exc} \tnote{,} \tnotex{t:omit} & $\forall x,y,z\in R$ &$\min_t\{\ord_{R/z^t R}(xy^t)\}\leq \varphi_R(\ord_R(x),\ord_{R/zR}(y))$ &\cite[Thm. 12.16]{Sch13}\\\hdashline
        \multirow{2}{*}{weakly normal}\tnotex{t:exc} \tnote{,} \tnotex{t:equi} & \multirow{2}{*}{$\forall p\in\N$ prime, $\forall y,z\in R$} &\multirow{1}{*}{$\max\left\{\min\left\{\ord_{R/z^2R}(y^2),\ord_{R/z^3R}(y^3)\right\},\min\left\{\ord_{R/z^pR}(y^p), \ord_{R/zR}(py)\right\}\right\}$}  & \multirow{2}{*}{\cref{thm: phi characterization for weak normality}}\\
        & & \multirow{1}{*}{$\leq \varphi_R(\ord_{R/zR}(y), \deg(z), p)$}&\\\hdashline
        seminormal\tnotex{t:exc} \tnote{,} \tnotex{t:equi} & $\forall y,z\in R$ & $\min\{\ord_{R/z^2R} (y^2), \ord_{R/z^3 R}(y^3)\}\leq \varphi_R(\ord_{R/zR}(y), \deg(z))$ & \cref{thm: phi characterization for seminormal}\\\hdashline
        \hline
        analytically unramified & \multirow{2}{*}{$\forall x\in R$}&\multirow{2}{*}{$\ord_R(x^2)\leq\varphi_R(\ord_R(x))$}& \multirow{2}{*}{\cite[Cor. 12.9]{Sch13}}\\\cdashline{1-1}
        reduced\tnotex{t:exc} & &  & \\\hdashline
        \multirow{2}{*}{analytically irreducible} &$\forall x,y\in R$ & $\ord_R(xy)\leq\varphi_R(\ord_R(x),\ord_R(y))$&\cite[Thm. 12.1]{Sch13} \\\cdashline{2-4}
        & $\forall x\in R$ &$\deg(x)\leq\varphi_R(\ord_R(x))$ & \cite[Cor. 12.10]{Sch13}\\\hdashline
        unmixed &$\forall x,y\in R$ & $\ord_R(xy)\leq \varphi_R(\deg(x),\ord_R(y))$ &\cite[Theorem 12.12]{Sch13} \\\hdashline
        quasi-unmixed & \multirow{2}{*}{$\forall x,y\in R$} & \multirow{2}{*}{$\ord_R\left((xy)^{\nu(R)}\right)\leq \varphi_R\left(\deg(x),\ord_R\left(y^{\nu(R)}\right)\right)$}& \multirow{2}{*}{\cite[Proposition 12.13]{Sch13}}\\\cdashline{1-1}
        equidimensional\tnotex{t:exc} & & &\\
        \hline 
        Cohen--Macaulay & $\forall \mathbf{x}$ s.o.p., $\forall \mathbf{y}$ $d$-tuple &$\ord_R(x_1y_1+\cdots+x_dy_d)\leq\varphi_R(\ell_R(R/\mathbf{x}R),\ord_{R/(x_1,\dots, x_{d-1})R}(y_d))$& \cite[Thm. 12.14]{Sch13}\\\hdashline
        \multirow{2}{*}{Generalized Cohen--Macaulay} & \multirow{1}{*}{$\exists N>0, \forall \sqrt{I}=\fm$,} &\multirow{2}{*}{$\ord_R(x_1y_1 + \dots + x_dy_d) \leq \varphi_R(\ell_{R}(R/(x_1, \dots, x_d)R), \ord_{R/(x_1, \dots, x_{d-1}) : \fm^N}(y_d))$}&\multirow{2}{*}{\cref{thm: phi characterization of generalized Cohen-Macaulay}}\\
        & \multirow{1}{*}{$\forall \mathbf{x}\in I$ s.o.p., $\forall \mathbf{y}$ $d$-tuple} &&\\\hdashline
        Buchsbaum & $\forall \mathbf{x}$ s.o.p., $\forall \mathbf{y}$ $d$-tuple&$\ord_R(x_1y_1+\cdots+x_dy_d)\leq\varphi_R(\ell_R(R/\mathbf{x}R),\ord_{R/(x_1,\dots, x_{d-1}):\fm}(y_d))$ & \cref{thm: phi characterization of Buchsbaum}\\
\bottomrule
\end{longtable}
\endgroup
 \begin{tablenotes}
    \tiny
    \note{$\nu(R)=\min\{t\mid \nil(R)^t=0\}$ denotes the nilpotency index of $R$}\\
    \item[e]\label{t:exc} $R$ is excellent\\
    \item[red]\label{t:red} $R$ is reduced\\
    \item[equi]\label{t:equi} $R$ is equidimensional\\
    \item[CM]\label{t:cm} $R$ is Cohen--Macaulay\\
    \item[1]\label{t:omit} In \cite{Sch13} Schoutens omits (but uses implicitly) the excellence hypothesis in the proofs of the listed results.
    \end{tablenotes}
\end{ThreePartTable}
\end{landscape}

\clearpage

\printbibliography
\end{document}